\documentclass{amsart}
\usepackage{amsfonts,amssymb,amscd,amsmath,enumerate,verbatim,calc}
\usepackage{amscd,amssymb,amsopn,amsmath,amsthm,graphics,amsfonts,enumerate,verbatim,calc}
\usepackage[dvips]{graphicx}
\usepackage[colorlinks=true,linkcolor=red,citecolor=blue]{hyperref}
\input xy
\xyoption{all}
\calclayout

\newcommand{\rt}{\rightarrow}
\newcommand{\lrt}{\longrightarrow}

\newcommand{\st}{\stackrel}

\newcommand{\La}{\Lambda}

\newcommand{\C}{\mathbb{C} }
\newcommand{\D}{\mathbb{D} }
\newcommand{\K}{\mathbb{K} }

\newcommand{\Z}{\mathbb{Z} }

\newcommand{\CA}{\mathcal{A} }
\newcommand{\CC}{\mathcal{C} }

\newcommand{\CF}{\mathcal{F} }

\newcommand{\CI}{\mathcal{I} }

\newcommand{\CP}{\mathcal{P} }

\newcommand{\CS}{\mathcal{S} }
\newcommand{\CT}{\mathcal{T} }
\newcommand{\CX}{\mathcal{X} }
\newcommand{\CY}{\mathcal{Y} }
\newcommand{\CW}{\mathcal{W}}
\newcommand{\CV}{\mathcal{V}}
\newcommand{\CU}{\mathcal{U}}

\newcommand{\BZ}{\mathbf{Z}}
\newcommand{\CB}{\mathcal{B} }

\newcommand{\CCF}{{\rm Cot}\mbox{-} \mathcal{F}}

\newcommand{\X}{\mathbf{X}}
\newcommand{\Y}{\mathbf{Y}}

\newcommand{\PP}{\mathbf{P}}

\newcommand{\Mod}{{\rm{Mod\mbox{-}}}}
\newcommand{\Modd}{{\rm{Mod}}_0\mbox{-}}

\newcommand{\mmod}{{\rm{{mod\mbox{-}}}}}

\newcommand{\Inj}{{\rm{Inj}\mbox{-}}}
\newcommand{\Prj}{{\rm{Prj}\mbox{-}}}
\newcommand{\prj}{{\rm{prj}\mbox{-}}}

\newcommand{\inj}{{\rm{inj}\mbox{-}}}

\newcommand{\im}{{\rm{Im}}}
\newcommand{\op}{{\rm{op}}}

\newcommand{\ac}{{\rm{ac}}}

\newcommand{\bb} {{\rm{b}}}

\newcommand{\Coker}{{\rm{Coker}}}
\newcommand{\Ker}{{\rm{Ker}}}

\newcommand{\pur}{{\rm pac}}

\newcommand{\Hom}{{\rm{Hom}}}
\newcommand{\Ext}{{\rm{Ext}}}
\newcommand{\End}{{\rm{End}}}

\theoremstyle{plain}
\newtheorem{theorem}{Theorem}[section]
\newtheorem{corollary}[theorem]{Corollary}
\newtheorem{lemma}[theorem]{Lemma}

\newtheorem{proposition}[theorem]{Proposition}

\theoremstyle{definition}
\newtheorem{definition}[theorem]{Definition}

\newtheorem{remark}[theorem]{Remark}

\newtheorem{s}[theorem]{}

\theoremstyle{plain}
\newtheorem{stheorem}{Theorem}[subsection]
\newtheorem{scorollary}[stheorem]{Corollary}
\newtheorem{slemma}[stheorem]{Lemma}
\newtheorem{sproposition}[stheorem]{Proposition}

\theoremstyle{definition}

\newtheorem{sremark}[stheorem]{Remark}

\newtheorem{sss}[stheorem]{}

\numberwithin{equation}{section}

\begin{document}

\title[Auslander's Formula: Variations and Applications]{Auslander's Formula: Variations and Applications}

\author[Asadollahi, Asadollahi, Hafezi, Vahed]{Javad Asadollahi, Najmeh Asadollahi, Rasool Hafezi and Razieh Vahed}

\address{Department of Mathematics, University of Isfahan, P.O.Box: 81746-73441, Isfahan, Iran and School of Mathematics, Institute for Research in Fundamental Sciences (IPM), P.O.Box: 19395-5746, Tehran, Iran }
\email{asadollahi@ipm.ir, asadollahi@sci.ui.ac.ir}

\address{Department of Mathematics, University of Isfahan, P.O.Box: 81746-73441, Isfahan, Iran}
\email{n.asadollahi@sci.ui.ac.ir}

\address{School of Mathematics, Institute for Research in Fundamental Sciences (IPM), P.O.Box: 19395-5746, Tehran, Iran }
\email{hafezi@ipm.ir}

\address{Department of Mathematics, Khansar Faculty of Mathematics and Computer Science, Khansar, Iran and School of Mathematics, Institute for Research in Fundamental Sciences (IPM), P.O.Box: 19395-5746, Tehran, Iran }
\email{vahed@ipm.ir}



\thanks{This research was in part supported by a grant from IPM (No: 94130216)}

\begin{abstract}
According to the Auslander's formula one way of studying an abelian category $\CC$ is to study $\mmod \CC$, that has nicer homological properties than $\CC$, and then translate the results back to $\CC$. Recently Krause gave a derived version of this formula and thus renewed the subject. This paper contains a detailed study of various versions of Auslander formula including the versions for all modules and for unbounded derived categories. We apply them to include some results concerning recollements of triangulated categories.
\end{abstract}

\maketitle

%\tableofcontents

\section{Introduction}
Let $\CC$ be an abelian category. A contravariant functor $F$ from $\CC$ to the category of abelian groups $\mathcal{A}b$ is called finitely presented, or coherent \cite{As1}, if there exists an exact sequence
$$ \Hom_{\CC}(-, X) \lrt \Hom_{\CC}(-,Y) \lrt F \lrt 0$$
of functors.
Let $\mmod \CC$ denote the category of all coherent functors. The systematic study of $\mmod \CC$ is initiated by Auslander \cite{As1}. He, not only showed that $\mmod \CC$ is an abelian category of global dimension less than or equal to two but also provided a nice connection between $\mmod \CC$ and $\CC$. This connection, which is known as Auslander formula \cite{L,K}, suggests that one way of studying $\CC$ is to study $\mmod \CC$, that has nicer homological properties than $\CC$, and then translate the results back to $\CC$. In particular, if we let $\CC$ to be $\mmod \La$, where $\La$ is an artin algebra, Auslander formula translates to the equivalence
$$\frac{\mmod (\mmod \La)}{\{F\mid F(\La)=0\}} \simeq \mmod \La$$
of abelian categories. As it is mentioned in \cite{L}, `a considerable part of Auslander's work on the representation theory of finite dimensional, or more general artin, algebras can be connected to this formula'.

Recently, Krause \cite{K} established a derived version of Auslander's formula, showing that $\D^{\bb}(\CC)$ is equivalent to a quotient of $\D^{\bb}(\mmod\CC)$. Also he gave a derived version of this formula for complexes of injective objects \cite[Sec. 4]{K}.

This work can be considered as a continuation of \cite{K}. It contains a detailed study of various versions of Auslander's formula, including the versions for all modules and for unbounded derived categories. These will have some applications, in particular, provide two expressions of $\D(\Mod R)$ as Verdier quotients. For the proof, we follow similar argument as in the proof of the classical case by Auslander, step by step. Let us be more precise on the structure of the paper.

Section \ref{Section 2} is the preliminary section and contains a collection of known facts that we need throughout the paper. Section \ref{Section 3} is devoted to Auslander's formula and its variations, from large Mod to different derived versions, i.e. unbounded, bounded above and bounded, both for contravariant and also covariant functors. One of the key points is a fundamental four terms exact sequence, similar to what Auslander has proved to exist \cite[pp. 203-204]{As1}. Here we use special flat resolutions instead of projective resolutions in Auslander's work, to prove that such a sequence exists in our context, see Proposition \ref{Aus-ExtSeq}. Let $R$ be a right coherent ring. We extend the existence of the sequence to complexes of functors over $\Mod(\mmod R)$ and apply it to present an unbounded derived version of Auslander's formula for $\Mod (\mmod R)$, i.e.
$$\frac{\D(\Mod (\mmod R))}{\D_0(\Mod (\mmod R))}\simeq \D(\Mod R),$$
where $\D_0(\Mod (\mmod R))$ is the thick subcategory of $\D(\Mod (\mmod R))$ consisting of all complexes $\X$ such that $\X(R)$ is an acyclic complex. This equivalence restricts to triangle equivalences
$$\frac{\D^*(\Mod (\mmod R))}{\D^*_0(\Mod (\mmod R))}\simeq \D^*(\Mod R),$$
where $* \in \{-, \bb\}$ and $\D^*_0(\Mod (\mmod R))= \D_0(\Mod (\mmod R)) \bigcap \D^*(\Mod (\mmod R))$. The argument works also to reprove Krause's result as well as its extension to unbounded derived categories, see Proposition \ref{Ext-Krause}. These are done in Subsection \ref{Contravariant functor}. A version of Auslander's formula for $\Mod(\mmod R)^{\op}$, the category of covariant functors from $\mmod R$ to $\CA b$ can be found in Subsection \ref{Covariant functors}.

In Section \ref{Section 4}, we apply our results to present two recollements and hence two descriptions of $\D(\Mod R)$ as the Verdier quotients of homotopy categories. To this end, we consider the pure-exact structure on the category $\Mod R$. The injective objects with respect to the pure-exact structure are called pure-injective $R$-modules. We denote the class of pure-injective $R$-modules  by ${\rm P}\Inj R$. Dually, we have the class ${\rm P}\Prj R$  of all pure-projective $R$-modules.
We show that the homotopy category $\K({\rm P}\Inj R)$ glues the homotopy categories $\K_{\ac}({\rm P}\Inj R)$ of all acyclic complexes over pure-injective $R$-modules and the derived category $\D(\Mod R)$, i.e. there is a recollement
\[\xymatrix@C=0.5cm{ \K_{\ac}({\rm P}\Inj R) \ar[rrr]^{ }   &&& \K({\rm P}\Inj R) \ar[rrr]^{} \ar@/^1.5pc/[lll]_{ } \ar@/_1.5pc/[lll]_{} &&& \D(\Mod R).\ar@/^1.5pc/[lll]_{} \ar@/_1.5pc/[lll]_{} }\]
Moreover, we show that similar recollement exists for $\K({\rm P}\Prj R)$.
There are some interesting consequences, among them an equivalence $$\K_{\ac}({\rm P} \Inj R) \simeq \K_{\ac}({\rm P}\Prj R),$$ of triangulated categories, where $\K_{\ac}({\rm P}\Inj R)$, resp. $\K_{\ac}({\rm P}\Prj R)$, denotes the homotopy category of all acyclic complexes of pure-injective, resp. pure-projective, $R$-modules.

Throughout the paper $R$ denotes a right coherent ring, $\Mod R$ denotes the category of all right $R$-modules and  $\mmod R$ denotes the full subcategory of $\Mod R$ consisting of all finitely presented modules.  \\

\section{Preliminaries}\label{Section 2}
In this section we collect some facts, that are needed throughout the paper.

\s Let $\CA$ be an abelian category. We denote by $\C(\CA)$ the category of all complexes over $\CA$ and by $\K(\CA)$ the homotopy category of $\CA$. Moreover, $\K^-(\CA)$, resp. $\K^{\bb}(\CA)$, denote the full subcategory of $\K(\CA)$ consisting of all bounded above, resp. bounded, complexes. The derived category of $\CA$ will be denoted by $\D(\CA)$. Moreover, $\D^{-}(\CA)$, resp. $\D^{\bb}(\CA)$, denotes the full subcategory of $\D(\CA)$ consisting of all homologically bounded above, resp. homologically bounded, complexes.

Let $\Prj\CA$, resp. $\Inj\CA$, denote the full subcategory of $\CA$ formed by all  projective, resp. all injective, objects. In case $\CA= \Mod R$, we abbreviate  $\Prj (\Mod R)$ to $\Prj R$ and  set $\prj R = \Prj R \bigcap \mmod R$. Similarly $\Inj R$ and $\inj R$ will be defined.\\

\noindent {\bf Functor categories.}
Let $\CC$ be an essentially small abelian category. The additive contravariant functors from $\CC$ to the category of abelian groups $\CA b$ together with the natural transformations between them form a category which is known as the functor category and is denoted either by $(\CC^{\op},\CA b)$ or $\Mod \CC$.
The category $\Mod \CC$, sometimes, is called the category of modules over $\CC$. It is known that $\Mod\CC$ is an abelian category. Similarly, all covariant functors and their natural transformations form an abelian category which is denoted by $(\CC, \CA b)$, or sometimes by $\Mod\CC^{\op}$.

It follows from Yoneda lemma that for every object $C \in \CC$, the representable functor $\Hom_{\CC}(-,C)$ is a projective object of $\Mod \CC$. Also, for every functor $F$ in $\Mod\CC$ there is an epimorphism  $\coprod_i \Hom_{\CC}(-,C_i) \lrt F \lrt 0$, where $C_i$ runs through all isomorphism classes of objects of $\CC$. Hence, the abelian category $\Mod \CC$ has enough projective objects.

A $\CC$-module $F$ is called finitely presented if there is the following  exact sequence
$$ \Hom_{\CC}(-,C_1) \lrt \Hom_{\CC}(-, C_0) \lrt F \lrt 0$$
of $\CC$-modules, where  $C_1 , C_0 \in \CC$. The category of all finitely presented $\CC$-modules is an abelian category \cite[Chapter III, \S 2]{Au2} and will be denoted by $\mmod \CC$.\\

\s \label{ProjObj}
Recall that a short exact sequence
$$0 \lrt M' \lrt M \lrt M''\lrt 0$$ of $R$-modules is called pure-exact if for every $N \in \mmod R$, the induced sequence $$ 0 \lrt \Hom_R(N, M') \lrt \Hom_R(N,M) \lrt \Hom_R(N,M'') \lrt 0$$
is exact. Consider the pure-exact structure on the category $\Mod R$.
An $R$-module $M$ is called pure-projective, if it is a projective object with respect to this exact structure. Warfield \cite{Wa} showed that pure-projective modules are precisely the direct summands of direct sums of finitely presented modules, see also \cite[33.6]{Wi}. We denote by ${\rm P}\Prj R$ the full subcategory of $\Mod R$ consisting of all pure-projective modules. The subcategory of $\Mod R$ consisting of pure injective modules, ${\rm P}\Inj R$, defines dually.
It is known that a functor $P$ in $\Mod (\mmod R)$ is projective if and only if $P \cong \Hom_R(-,M)$ for some pure-projective $R$-module $M$, see e.g. \cite[Theorem B.10]{JL}.\\

\noindent {\bf Recollements of abelian categories.}
A subcategory $\CC$ of an abelian category $\CA$ is called a Serre subcategory, if for every short exact sequence $0 \rt C_1 \rt C_2 \rt C_3 \rt 0$  in $\CA$, $C_2 \in \CC$ if and only if $C_1, C_3 \in \CC$.
For a Serre subcategory $\CC$ of $\CA$, Gabriel \cite{Gab} constructed  an abelian category $\CA/\CC$ with the same objects as in $\CA$ and  morphism sets
\[ \Hom_{\CA/\CC}(X, Y) = \underset{ X',Y'}{\underrightarrow{\rm lim}} \Hom_{\CA}(X', Y/Y'),\]
where $X'$, resp. $Y'$, runs through all subobjects of $X$, resp. $Y$, such that $X/X'$, resp. $Y'$, lies in $\CC$.
Assigned to a Serre subcategory $\CC$ of $\CA$, there is an exact and dense quotient functor
$Q: \CA \lrt \CA/\CC$. A Serre subcategory $\CC$ is called localizing, resp. colocalizing, if $Q$ possesses  a right, resp. left, adjoint.

Let $\CB$ be another abelian category and $F: \CA \lrt \CB$ be an additive functor. We set
$ \im F= \{ B \in \CB \mid B \cong F(A), \text{ for some } A \in \CA\}$
and
$\Ker F= \{ A \in \CA \mid F(A)=0\}.$

\begin{definition}\label{Def-Rec}
Let $\CA'$, $\CA$ and $\CA''$ be abelian categories. A recollement \cite{BBD} of $\CA$ with respect to $\CA'$ and $\CA''$ is a diagram
\[\xymatrix{\CA'\ar[rr]^{i_*=i_!}  && \CA \ar[rr]^{j^*=j^!} \ar@/^1pc/[ll]_{i^!} \ar@/_1pc/[ll]_{i^*} && \CA'' \ar@/^1pc/[ll]_{j_*} \ar@/_1pc/[ll]_{j_!} }\]
of additive functors satisfying the following conditions:
\begin{itemize}
\item[$(i)$] $(i^*,i_*)$, $(i_!,i^!)$, $(j_!, j^!)$ and $(j^*,j_*)$ are adjoint pairs.
\item[$(ii)$] $i_*$, $j_*$ and $j_!$ are fully faithful.
\item[$(iii)$] $\im i_*= \Ker j^*$.
\end{itemize}
\end{definition}

A sequence of abelian categories is called a localization sequence if the lower two rows of a recollement exist and the functors appearing in these two rows, i.e. $i_*, i^!, j^!$ and $j_*$, satisfy all the conditions in the definition above which involve only these functors. Similarly, one can define a colocalization sequence of abelian categories via the upper two rows.\\

For the proof of the following facts see e.g. \cite{FP} and \cite{Gab}.

\begin{remark}\label{Properties}  Consider the recollement of Definition \ref{Def-Rec}. Then the functors $i_*$ and $j^*$ are exacts and $i_*$ induces an equivalence between $\CA'$ and the Serre subcategory $\im i_* = \Ker j^*$ of $\CA$. In particular, $\CA'$ can be considered as a Serre subcategory of $\CA$. Furthermore, since the exact functor $j^*$ has a fully faithful right, resp. left, adjoint, $\CA'$ is a localizing, resp, colocalizing, subcategory of $\CA$  and there exists an equivalence $\CA'' \simeq \CA/\CA'$.
\end{remark}

\noindent {\bf Recollements of triangulated categories and stable $t$-structures.} Let $\CT$, $\CT'$ and $\CT''$ be triangulated categories.

\begin{definition}
A recollement of  $\CT$  relative to  $\CT'$ and $\CT''$ is defined by six triangulated functors as follows
\[\xymatrix{\CT'\ar[rr]^{i_*=i_!}  && \CT \ar[rr]^{j^*=j^!} \ar@/^1pc/[ll]_{i^!} \ar@/_1pc/[ll]_{i^*} && \CT'' \ar@/^1pc/[ll]_{j_*} \ar@/_1pc/[ll]_{j_!} }\]
satisfying the following conditions:
\begin{itemize}
\item[$(i)$] $(i^*,i_*)$, $(i_!,i^!)$, $(j_!, j^!)$ and $(j^*,j_*)$ are adjoint pairs.
\item[$(ii)$] $i^!j_*=0$, and hence $j^!i_!=0$ and $i^*j_!=0$.
\item[$(iii)$] $i_*$, $j_*$ and $j_!$ are fully faithful.
\item[$(iv)$] for any object $T \in \CT$, there exist the following triangles
\[i_!i^!(T) \rt  T \rt j_*j^*(T) \rightsquigarrow \ \ \ \text{and} \ \ \ j_!j^!(T) \rt T \rt i_*i^*(T) \rightsquigarrow\]
in $\CT$.
\end{itemize}
\end{definition}

Similar to the case of abelian categories, one can define a localization and a colocalization sequence of triangulated categories.

\begin{definition}
A pair $(\CU, \CV)$ of full  subcategories of a triangulated category $\CT$ is called a stable $t$-structure in  $\CT$ if the following conditions are satisfied.
\begin{itemize}
\item[$(i)$] $\mathcal{U}=\Sigma \mathcal{U}$ and $\mathcal{V}= \Sigma \mathcal{V}$.
\item[$(ii)$] $\Hom_{\CT}(\mathcal{U}, \mathcal{V})=0$.
\item[$(iii)$] For each $X \in \CT$, there is a triangle $U \rt X \rt V \rightsquigarrow$ with $U \in \mathcal{U}$ and $V \in \mathcal{V}$.
\end{itemize}
\end{definition}

Following result establishes a close relation between recollements of triangulated categories and stable $t$-structures, see \cite[Proposition 2.6]{Mi}.

\begin{proposition} \label{Miyachi} Let $\CT$ be a triangulated category. Let $(\CU,\CV)$ and $(\CV,\CW)$ be stable $t$-structures in $\CT$. Then there is a recollement
\[\xymatrix@C-0.5pc@R-0.5pc{ \CV  \ar[rr]^{i_*}  && \CT \ar[rr]^{j^*} \ar@/^1pc/[ll]_{i^!} \ar@/_1pc/[ll]_{i^*} && \CT/\CV \ar@/^1pc/[ll]_{j_*} \ar@/_1pc/[ll]_{j_!} }\]
in which  $i_*: \CV \lrt \CT$ is an inclusion functor, $\im j_!=\CU$ and $\im j_* =\CW$.
\end{proposition}

We also need the following result of Miyachi.

\begin{proposition}\label{Miyachi2} \cite{Mi} Let $\CT$ be a triangulated category.
 Then the following statements hold true.
\begin{itemize}
\item [$(i)$]Let $(\CU,\CV)$ be a stable $t$-structure in $\CT$. Then the inclusion functor $i_*: \CU \lrt  \CT$, resp.\ $j_*: \CV \lrt \CT$, admits a right adjoint $i^!: \CT \lrt \CU$ , resp.\  a left adjoint $j^*: \CT \lrt \CV$. Moreover,  the functor $i^!$, resp.\ $j^*$, induces a triangle  equivalence $\CT/\CV \simeq \CU$, resp.\ $\CT/\CU \simeq \CV$.
\item [$(ii)$] If the inclusion functor $i_*: \CT' \lrt \CT$ has a left adjoint $i^*: \CT \lrt \CT'$, then $(\Ker i^*, \im i_*)$ is a stable $t$-structure in $\CT$.
\item [$(iii)$] If the inclusion functor $i_*: \CT' \lrt \CT$ has a right adjoint $i^*: \CT \lrt \CT'$, then $(\im i_*, \Ker i^*)$ is a stable $t$-structure in $\CT$.
\end{itemize}
\end{proposition}

\noindent {\bf Cotorsion theory.} A pair $(\CX, \CY)$ of classes of objects of an abelian category $\CA$ is called a cotorsion theory if $\CX^\perp=\CY$ and $\CX = {}^\perp\CY$, where the left and right orthogonals are taken with respect to $\Ext^1_{\CA}$. So for example
\[\CX^\perp:=\{A \in \CA \ \mid \ \Ext^1_{\CA}(X,A)=0, \ {\rm{for \ all}} \ X \in \CX \}.\]
A cotorsion theory $(\CX, \CY)$ is called complete if for every $A \in \CA$ there exist exact sequences
$0 \rt Y \rt X \rt A \rt 0$ and $0 \rt A \rt Y' \rt X' \rt 0$, with $X, X'\in \CX$ and $Y, Y' \in \CY$.

A cotorsion theory $(\CX, \CY)$ is said to be cogenerated by a set if there is a set $\CS$ of objects of $\CA$ such that $\CS^\perp = \CY$. If a cotorsion theory is cogenerated by a set, then it is complete, see \cite[Theorem 10]{ET}.

\section{Deriving Auslander's formula}\label{Section 3}
In this section, we prove a version of Auslander's formula for $\Mod R$ and also provide an unbounded derived version of Auslander's formula for it, i.e. we prove that $\D(\Mod R)$ is equivalent to a quotient of $\D(\Mod(\mmod R))$. Note that in this section $R$ is always a right coherent ring.

\subsection{Contravariant functors}\label{Contravariant functor}
Recall that a functor $F \in \Mod (\mmod R)$ is called flat if every morphism $f: E \lrt F$, where $E$ is a finitely presented functor, factors through a representable functor $P$.  It is known that an object $F$ of $\Mod (\mmod R)$ is flat if and only if $F \cong \Hom_R(-, M)$, for some $R$-module $M$, see \cite[Theorem B.10]{JL}.
Let $\CF(\mmod R)$ denote the full subcategory of $\Mod (\mmod R)$ consisting of all flat objects.

It is known \cite[Theorem B.11]{JL} that the functor
$$U: \Mod R \lrt \Mod (\mmod R)$$
given by $U(M) = \Hom_R(-, M)|_{\mmod R}$ is fully faithful and induces an equivalence $\Mod R \simeq \CF(\mmod R)$ of categories.

For simplicity, throughout we write $(-, M)$, resp. $(-, f)$, instead of $\Hom_R(-, M)|_{\mmod R}$, resp. $\Hom_R(-, f)$, where $M$ is an $R$-module and $f$ is an $R$-homomorphism.

Let $F$ be an object in $\Mod (\mmod R)$. Consider the following projective presentation of $F$
$$ (-,M_1) \st{(-,d_1)}\lrt (-,M_0) \lrt F \lrt 0, $$
where $M_1$ and $M_0$ belong to ${\rm P}\Prj R$.
Let $M_2 \st{d_2} \lrt M_1$ be the kernel of $d_1$. Then there is a flat resolution
$$ 0 \lrt (-,M_2) \st{(-,d_2)} \lrt (-,M_1) \st{(-,d_1)} \lrt (-,M_0) \lrt F \lrt 0$$
of $F$. Hence, every functor in $\Mod(\mmod R)$ has a flat resolution of length at most 2.

By definition, for a functor $F \in \Mod (\mmod R)$, a flat precover is a morphism $\pi:(-, M) \lrt F$ such that $M \in \Mod R$ and $\Hom (- , (-,M)) \lrt \Hom (-, F)$ is surjective on $\CF(\mmod R)$. If, moreover, the kernel of $\pi$ belongs to $\CF(\mmod R)^\perp$, then $\pi: (-,M) \lrt F$ is called a special flat precover, where orthogonal is taken with respect to the functor $\Ext^1$.

\begin{sremark}
A flat resolution
$$ 0 \lrt (-,M_2) \st{(-,d_2)} \lrt (-,M_1) \st{(-,d_1)} \lrt (-,M_0) \st{\varepsilon}\lrt F \lrt 0$$
of $F$ is called special, if both morphisms $(-,M_0) \st{\varepsilon}\lrt F$ and $(-,M_1) \st{(-,d_1)} \lrt \Ker \varepsilon$ are special flat precovers. It is shown by Herzog \cite[Proposition 7]{H} that every functor in $\Mod (\mmod R)$ admits a special flat resolution of length at most 2.
\end{sremark}

To prepare the ground for our main result, we follow, step by step, Auslander's argument in Sections 2 and 3 of \cite{As1}. Since the techniques are similar, we just explain the outlines. The details then are straightforward and can be found in \cite{As1}.\\

Consider the embedding $U : \Mod R \lrt \Mod (\mmod R)$. It induces a functor
$$U_{\centerdot}: (\Mod (\mmod R) , \Mod R) \lrt (\Mod R, \Mod R).$$

Let $F $ be an object in $(\Mod R , \Mod R)$. Then for a functor $G \in \Mod (\mmod R)$, set
$$ (U^{\centerdot}F)(G):= \Coker (F(M_1) \st{F(d_1)} \lrt F(M_0)),$$
where $$0 \lrt (-,M_2) \st{(-, d_2)}\lrt (-, M_1) \st{(-, d_1)} \lrt (-, M_0) \lrt G \lrt 0$$ is a special flat resolution of $G$.
Moreover, a map $g: G \lrt G'$ in $\Mod(\mmod R)$ can be lifted to a map of their special flat resolutions. This, in turn, induces a map $U^{\centerdot}F(g): U^{\centerdot}F(G) \lrt U^{\centerdot}F(G')$. Known techniques in homological algebra guaranteeing that the definitions of $(U^{\centerdot}F)(G)$ and $(U^{\centerdot}F)(g)$ are independent of the choice of special flat resolutions of $G$ and $G'$. In fact $U^{\centerdot}F$ is a functor. It is straightforward to check that $(U^{\centerdot}F, U_{\centerdot})$ is an adjoint pair and similar to Proposition 2.1 of \cite{As1}, we have the following. We skip the details of the proof.

\begin{sproposition}\label{leftaadjoint}
Consider the adjoint pair $(U^{\centerdot}F, U_{\centerdot})$. For a functor $F \in (\Mod R, \Mod R)$, $U_{\centerdot} U^{\centerdot} F=F$. Moreover, $U^{\centerdot}F$ is an exact functor if $F$ is so.
\end{sproposition}

Let $i: \Mod R \lrt \Mod R$ be the identity functor and consider the functor $$\nu=U^{\centerdot}i: \Mod(\mmod R) \lrt \Mod R.$$ It follows from the above properties that if
$$  0 \lrt (-,M_2) \st{(-,d_2)} \lrt (-,M_1) \st{(-,d_1)} \lrt (-,M_0) \lrt F \lrt 0$$
is a special flat resolution of $F$ in $\Mod(\mmod R)$, then $\nu(F)= \Coker(M_1 \st{d_1} \lrt M_0)$.

\sss Let $\Modd(\mmod R)$ be the full subcategory of $\Mod(\mmod R)$ consisting of all functors $F$ such that $\nu(F)=0$. Note that since $i$ is an exact functor, $\nu$ is exact. So, $\Modd(\mmod R)$ is a Serre subcategory of $\Mod(\mmod R)$, see Sec. 3 of \cite{As1}.

\begin{sremark}
A functor $F \in \Mod (\mmod R)$ is called effaceable if $F(P)=0$, for every finitely generated projective $R$-module $P$, see \cite[p. 141]{G}. Also, by definition of the functor $U^{\centerdot}$, a functor $F \in \Mod(\mmod R)$ belongs to $\Modd(\mmod R)$ if and only if $F(R)=0$, or equivalently $F(P)=0$, for every finitely generated projective $R$-module $P$. This fact identifies $\Modd(\mmod R)$ with the effaceable functors in $\Mod (\mmod R)$.
\end{sremark}

Let $F \in \Mod(\mmod R)$. If we consider a special flat resolution of $F$, a similar argument as in \cite[pp. 203-204]{As1} can be applied now, to prove the following result.
Since this sequence plays a central role in the paper, we include a brief proof.

\begin{sproposition}\label{Aus-ExtSeq}
For each functor $F $ in $\Mod (\mmod R)$ there is an exact sequence
$$ 0 \lrt F_0 \lrt F \lrt ( -, \nu(F)) \lrt F_1 \lrt 0$$
such that $F_0, F_1 \in \Modd(\mmod R)$.
\end{sproposition}

\begin{proof}
Consider a special flat resolution
$$  0 \lrt (-,M_2) \st{(-,d_2)} \lrt (-,M_1) \st{(-,d_1)} \lrt (-,M_0) \lrt F \lrt 0$$
of $F$.
By definition $\nu(F)= \Coker (M_1 \st{d_1} \lrt M_0)$. If we set $M := \Coker (M_2 \st{d_2} \lrt M_1)$, then there are short exact sequences
\[ 0 \lrt M_2 \lrt M_1 \lrt M \lrt 0 \ \ \text{and} \ \ 0 \lrt M \lrt M_0 \lrt \nu(F) \lrt 0.\]
So there is the following commutative diagram
\[ \xymatrix@C-0.5pc@R-0.5pc{ &&& 0 \ar[d] &0 \ar@{.>}[d] & \\ 0 \ar[r] & (-, M_2) \ar[r] \ar@{=}[d] & (-,M_1) \ar[r] \ar@{=}[d] & (-,M) \ar[r] \ar[d] & F_0 \ar@{.>}[d] \ar[r] & 0 \\
0\ar[r] & (-,M_2) \ar[r] & (-, M_1) \ar[r] & (-,M_0)\ar[d] \ar[r] & F \ar[r] \ar@{.>}[dl] & 0 \\
&&& (-, \nu(F)) \ar[d] & & \\ &&& F_1 \ar[d] &&\\
&&& 0 && }\]
In the above diagram $F_0$, resp. $F_1$, is a  cokernel of $((-, M_1) \lrt (-,M))$, resp.  $((-,M_0) \lrt (-,\nu(F))$, and so belongs to $\Modd(\mmod R)$. So, by the Snake Lemma, there is a unique exact sequence, identified by dashed line,
$$0 \lrt F_0 \lrt F \lrt (-, \nu(F)) \lrt F_1 \lrt 0$$
making the above diagram commutative. This completes the proof.
\end{proof}

\begin{sremark}\label{Hom&Ext}
Let $F$ be a functor in $\Modd (\mmod R)$. Then, similar to Proposition 3.2 of \cite{As1}, we can deduce that $\Ext^i(F, (-,M))=0$, for $i=0,1$ and for all $M \in \Mod R$. We do not include the proof as it follows by similar argument.
\end{sremark}

As a direct consequence of Proposition \ref{Aus-ExtSeq} and Remark \ref{Hom&Ext}, we have the following theorem. It can be proven using Proposition II.2 of \cite{Gab}, but by a completely different approach.

\begin{stheorem}(Compare \cite[Theorem 2.3]{K})\label{Loc-Seq}
There exists a localization sequence
\[\xymatrix@C=0.5cm{ \Modd (\mmod R) \ar[rrr]  &&& \Mod (\mmod R)  \ar[rrr]^{\nu} \ar@/^1.5pc/[lll] &&& \Mod R  \ar@/^1.5pc/[lll]^{U} }\]
of abelian categories. In particular, $\Mod R \simeq \frac{\Mod (\mmod R)}{\Modd(\mmod R)}$.
\end{stheorem}

\begin{proof}
Let $F$ be a functor in $\Mod (\mmod R)$. Then by Proposition \ref{Aus-ExtSeq} there is an exact sequence
$$ 0 \lrt F_0 \lrt F \lrt (-, \nu(F)) \lrt F_1 \lrt 0$$
with $F_0 , F_1 \in \Modd (\mmod R)$.
In view of Remark \ref{Hom&Ext}, $\Ext^i (F_j, (-,M)) =0$, for $i,j \in \{0, 1\}$ and for every module $M \in \Mod R$. This fact, in view of the fully faithfulness of the functor $U$ imply the following natural isomorphisms
\[\begin{array}{ll}
\Hom(F, (-,M)) & \cong \Hom((-, \nu(F)), (-,M))\\
& \cong \Hom_R(\nu(F), M).
\end{array}\]
Thus $\nu: \Mod (\mmod R) \lrt \Mod R$ provides a left adjoint of $U: \Mod R \lrt \Mod (\mmod R)$. Hence the existence of the desired localization sequence follows from Lemma 2.1 of \cite{K97}.
\end{proof}

\sss Let $\K^*_{R \mbox{-} \ac}(\Mod (\mmod R))$ denote the full subcategory of $\K^*(\Mod(\mmod R))$ consisting of all complexes $(\X,d^i)$ such that $$\X(R): \ \cdots \lrt X^{i-1}(R) \st{d^{i-1}(R)} \lrt X^i(R) \st{d^i(R)} \lrt X^{i+1}(R) \lrt \cdots$$ is an acyclic complex of abelian groups, where $* \in \{\text{blank}, -, \bb\}$. Note that if $\X$ is a complex in $\K^*_{R\mbox{-} \ac}(\Mod (\mmod R))$, then $\X(P)$ is acyclic for all $P \in \prj R$.

It can be easily checked that $\K^*_{R\mbox{-} \ac}(\Mod (\mmod R))$ is a thick subcategory of $\K^*(\Mod(\mmod R))$. So we can form the quotient category $$\K^*(\Mod (\mmod R))/ \K^*_{R \mbox{-} \ac}(\Mod (\mmod R)),$$  which we denote it by $\D^*_{R}(\Mod (\mmod R))$.\\

In our next theorem we show that there is a triangle equivalence between $\D(\Mod R)$ and $\D_{R}(\Mod (\mmod R))$. This fact has a short proof using a `ring with several objects' version of Corollary 3.6 of \cite{AHV}. There it is proved that for an artin algebra $\La$ and an $R$-module $M \in \mmod\La$, there is an equivalence of triangulated categories $\D^{\bb}_M(\Mod\La)\simeq \D^{\bb}(\Mod\End_{\La}(M))$. It can be generalized to the functor categories without severe problems. So deriving the equivalence of Theorem \ref{Loc-Seq} in view of this fact, implies the equivalence mentioned above. Here we present a constructive proof, introducing the equivalence map as we need its exact definition in our next results. To this end, we need the following lemma that provides a complex version of Proposition \ref{Aus-ExtSeq} and Remark \ref{Hom&Ext}.

\begin{slemma}\label{ExSeqCom} \label{Hom0}
$(i)$ Let $\X \in \C(\Mod (\mmod R))$. There exists an exact sequence
$$ 0 \lrt \X_0 \lrt \X \lrt (-, \nu(\X)) \lrt \X_1 \lrt 0,$$
where $\X_0 $ and $\X_1$ are complexes over $\Modd(\mmod R)$ and $\nu(\X)$ is a complex over $\Mod R$ whose i-th degree is $\nu(X^i)$.\\

\noindent $(ii)$ Let ${\bf F} \in \Modd(\mmod R)$. Then $\Ext^i({\bf F},(-, {\bf C}))=0$, for $j=0, 1$ and for every ${\bf C} \in \C(\Mod (\mmod R))$.
\end{slemma}

\begin{proof}
$(i)$ Let $(\X,d^i)$ be a complex over $\Mod(\mmod R)$. By Proposition \ref{Aus-ExtSeq}, for every $i \in \Z$, there is an exact sequence
$$ 0 \lrt X^i_0 \lrt X^i \lrt ( -, \nu(X^i)) \lrt X^i_1 \lrt 0,$$
such that $X^i_0$ and $X^i_1$ belong to $\Modd (\mmod R)$.
In view of Lemma \ref{Hom&Ext}, for $j=0, 1$, $$\Ext^j_R(X_0^i, (-, \nu(X^{i+1})))=0=\Ext^j_R(X_1^i, (-, \nu(X^{i+1}))),$$ for all $i \in \Z$. Thus, for every $i \in \Z$ there is a unique morphism $\overline{d^i}: \nu(X^i) \lrt \nu(X^{i+1})$ making the following diagram commutative
\[\xymatrix{ X^i \ar[r] \ar[d]^{d^i} & (-, \nu(X^i))  \ar[d]^{(-,\overline{d^i})} \\
 X^{i+1} \ar[r] & (-, \nu(X^{i+1})).}\]
Moreover, there exist unique morphisms $d_0^i: X_0^i \lrt X_0^{i+1}$ and $d_1^i: X_1^i \lrt X_1^{i+1}$ such that the following diagram
\[\xymatrix{
0 \ar[r] & X_0^i \ar[r] \ar[d]^{d_0^i} & X^i \ar[r] \ar[d]^{d^i} & (-, \nu(X^i)) \ar[r] \ar[d]^{(-,\overline{d^i})} & X_1^i \ar[r] \ar[d]^{d_1^i} & 0 \\
0 \ar[r] & X_0^{i+1} \ar[r] & X^{i+1} \ar[r] & (-, \nu(X^{i+1})) \ar[r] & X_1^{i+1} \ar[r] & 0}\]
is commutative.
The uniqueness of $d^i_0$, $d^i_1$ and $\bar{d^i}$ yield the existence of complexes
$$\X_0 : \cdots \lrt X^{i-1}_0 \st{d_0^{i-1}} \lrt X^i_0 \st{d^i_0} \lrt X^{i+1}_0 \lrt \cdots, \ \ \ \ \ \ \ $$
$$\X_1 : \cdots \lrt X^{i-1}_1 \st{d_1^{i-1}} \lrt X^i_1 \st{d^i_1} \lrt X^{i+1}_1 \lrt \cdots, \ \ \text{and}$$
$$\nu(\X) : \cdots \lrt \nu(X^{i-1}) \st{\overline{d^{i-1}}} \lrt \nu(X^i) \st{\overline{d^i}} \lrt \nu(X^{i+1}) \lrt \cdots$$
such that fit in the following exact sequence of complexes
$$ 0 \lrt \X_0 \lrt \X \lrt (-, \nu(\X)) \lrt \X_1 \lrt 0.$$
Hence, we get the desired exact sequence.

$(ii)$ Note that $\C(\Mod(\mmod R))$ is an abelian category with enough projectives \cite[Theorem 10.43]{Rot}. Moreover, by Theorem 10.42 of \cite{Rot}, projective complexes are split exact complexes of projectives. Let $\PP$ be a projective complex. Since $U: \Mod R \lrt \Mod(\mmod R)$ is full and faithful, in view of \ref{ProjObj}, a projective complex $\PP$ is isomorphic to the complex
$$ (-, {\bf M}): \ \ \cdots \lrt (-, M^{i-1}) \lrt (-,M^i) \lrt (-,M^{i+1}) \lrt \cdots, $$
where for all $i$, $M^i$ is a pure-projective $R$-module.

Since $(-, {\bf C})$ is left exact, to prove the result, it is enough to show that $\Hom({\bf F}, (-, {\bf C}))=0$, see Proposition 3.2 of \cite{As1}. To see this, asume that $\tau$ is a map in $\Hom({\bf F}, (-, {\bf C}))$. Then, we have a map $\tau \circ \pi: (-, {\bf M}) \lrt (-, {\bf C})$ of complexes over $\Mod (\mmod R)$. The fully faithful functor $U: \Mod R \lrt \Mod(\mmod R)$ can be naturally extended to the full and faithful functor $\C(U): \C(\Mod R) \lrt \C( \Mod (\mmod R))$. Thus, $\tau \circ \pi = (-, f)$, where $f: {\bf M} \lrt {\bf C}$ is a chain map of complexes.
Since ${\bf F}$ is a complex over $\Modd (\mmod R)$,
$${\bf F}(R): \ \cdots \lrt F^{i-1}(R) \lrt F^i(R) \lrt F^{i+1}(R) \lrt \cdots$$
is a zero complex. So $\tau$ vanishes in $\C(\Mod (\mmod R))$.
\end{proof}

\begin{stheorem}\label{Thm1}
There exists an equivalence
$$\D(\Mod R) \simeq \D_{R}(\Mod (\mmod R))$$
of triangulated categories, where as in our setup $R$ is a right coherent ring.
\end{stheorem}

\begin{proof}
Define the functor $\eta: \D(\Mod R) \lrt \D_{R}(\Mod (\mmod R))$ by the attachment $\eta(\X)=(-,\X)$. Also, $\eta$ maps every roof $\xymatrix{\X  & \ar[l]_s\BZ \ar[r]^f & \Y }$ in $\D(\Mod R)$ to the roof
$$\xymatrix{ (-,\X) & \ar[l]_{(-,s)} (-,\BZ) \ar[r]^{(-,f)}& (-,\Y) }$$ in $\D_R(\Mod (\mmod R))$. Observe that since ${\rm cone}(s)$ belongs to $\K_{\ac}(\Mod R)$, ${\rm cone}((-,s))$ belongs to $\K_{R\mbox{-} \ac}(\Mod (\mmod R))$.

We claim that $\eta$ is faithful, full, dense and also a triangle functor. Let $\xymatrix{\X & \ar[r]^f\BZ\ar[l]_s  & \Y }$ be a roof in $\D(\Mod R)$ such that the induced roof $\xymatrix{(-, \X) & (-, \BZ) \ar[l]_{(-,s)} \ar[r]^{(-,f)}& (-,\Y) }$ is zero in $\D_{R}(\Mod (\mmod R))$. So, there is a morphism $T: {\bf F} \lrt (-,\BZ) $ such that ${\rm cone}(T) \in \K_{R\mbox{-} \ac}(\Mod (\mmod R))$ and $ (-, f) \circ T =0$ in $\K(\Mod (\mmod R))$. In view of Lemma \ref{ExSeqCom}
there exists an exact sequence
\[\xymatrix{ 0 \ar[r] & {\bf F}_0 \ar[r]& {\bf F} \ar[r]^{W \ \ \ \ \ \  }  & (-, \nu({\bf F})) \ar[r]  &  {\bf F}_1\ar[r] &0,}\]
such that ${\bf F}_0(R)=0 ={\bf F}_1(R)$ and $\nu({\bf F})$ is a complex whose $i$-th degree is $\nu(F^i)$. Since,  $W(R): {\bf F}(R) \lrt \nu({\bf F})$ is an isomorphism, there is a map $\nu({\bf F}) \st{W(R)^{-1}} \lrt {\bf F}(R) \st{T(R)} \lrt \BZ$ with acyclic cone, such that $(T\circ X^{-1})(R) \circ f=0$ in $\K(\Mod R)$.  Hence, the roof  $\xymatrix{\X & \BZ \ar[l]_s\ar[r]^f & \Y }$ is zero in $\D(\Mod R)$. Thus, $\eta$ is faithful.

To see that $\eta$ is full, let $$\xymatrix{(-,\X)  & {\bf H} \ar[r]^f \ar[l]_s & (-, \Y) }$$ be a roof in $\D_R(\Mod (\mmod R))$. By Lemma \ref{ExSeqCom}, there exist an exact sequence
$$ 0 \lrt {\bf H}_0 \lrt {\bf H} \st{\varphi}\lrt (-,\nu({\bf H})) \lrt {\bf H}_1 \lrt 0,$$
where ${\bf H}_0$ and ${\bf H}_1$ are complexes over $\Modd (\mmod R)$ and $\nu({\bf H})$ is a complex whose $i$-th degree is $\nu(H^i)$. By Lemma \ref{Hom0}, $\Ext^j({\bf H}_0 , (-, {\bf C})) =0 =\Ext^j({\bf H}_1, (-,{\bf C}))$ for $j=0, 1$ and every complex ${\bf F}$ over $\Modd (\mmod R)$. So, by applying the Hom functors $\Hom(-, (-, \X))$ and $\Hom (-, (-,\Y))$ on the above exact sequence, respectively, we have isomorphisms
$$ \Hom({\bf H}, (-, {\bf X})) \cong \Hom((-, \nu({\bf H})), (-, \X)) \ \ \text{and}$$
$$ \Hom({\bf H}, (-, \Y)) \cong \Hom ((-, \nu({\bf H})), (-,\Y)).$$
Therefore, there are morphisms $\varphi_{\X}: (-, \nu({\bf H})) \lrt (-, \X)$ and $\varphi_{\Y}: (-, \nu({\bf H})) \lrt (-,\Y)$  such that $\varphi_{\X} \circ \varphi=s$ and $\varphi_{\Y}\circ \varphi=f$. Note that since ${\rm cone}(s) \in \K_{R\mbox{-} \ac}(\Mod (\mmod R))$ and $\varphi(R)$ is an isomorphism, ${\rm cone}(\varphi_{\X})$ belongs to $\K_{R\mbox{-} \ac}(\Mod (\mmod R))$.

Now, the commutative diagram
\[\xymatrix@C-0.8pc@R-0.5pc{ & & {\bf H} \ar[dr]^{\varphi} \ar[dl]_{\rm id}& & \\
 & {\bf H} \ar[dl]_s \ar[drrr]^{ f  \ \ } && (-, \nu({\bf H}))\ar[dr]^{\varphi_{\Y}} \ar[dlll]_{\varphi_{\X}} &  \\ (-, \X)  & & & & (-,\Y)   }\]
implies that the roof $\xymatrix{ (-,\X)  & {\bf H} \ar[r]^f \ar[l]_{\ \ \ s } & (-, \Y) }$ is equivalent to the roof $$\xymatrix{(-,\X)  & (-, \nu({\bf H}))\ar[r]^{\varphi_{\X}} \ar[l]_{  \varphi_{\Y}}& (-, \Y) } $$ in $\D_R(\Mod (\mmod R))$.

Moreover, let $\X$ be a complex in $\D_{R}(\Mod (\mmod R))$. Then an exact sequence
$$0 \lrt \X_0 \lrt \X \lrt (-, \nu(\X)) \lrt \X_1 \lrt 0$$
implies that $\X$ is isomorphic to $(-, \nu(\X))$ in $\D_{R}(\Mod (\mmod R))$  and so $\eta$ is dense.

Finally, it is obvious that $\eta$ is a triangle functor.
\end{proof}

Set $\D_0^* (\Mod (\mmod R)):= \frac{\K^*_{R\mbox{-} \ac}(\Mod (\mmod R))}{\K^*_{\ac}(\Mod (\mmod R))}$ and $\D_0^* (\mmod (\mmod R)):= \frac{\K^*_{R \mbox{-} \ac}(\mmod (\mmod R))}{\K^*_{\ac}(\mmod (\mmod R))}$, where $* \in \{ {\rm blank}, -, \bb\}$. It follows from \cite[Corollaire 4-3]{V2} that there is the following equivalences of triangulated categories
$$ \D^*_{R}(\Mod (\mmod R)) \simeq \frac{\D^*(\Mod (\mmod R))}{\D_0^* (\Mod (\mmod R))},$$
$$ \D^*_{R}(\mmod (\mmod R)) \simeq \frac{\D^*(\mmod (\mmod R))}{\D_0^* (\mmod (\mmod R))}.$$

Following corollary is a derived version of Auslander's formula for $\Mod (\mmod R)$.

\begin{scorollary}
There is a commutative diagram
\[ \xymatrix{ \D(\Mod R) \ar[rr]^{\st{\eta}\sim} && \frac{\D(\Mod (\mmod R))}{\D_0(\Mod (\mmod R))} \\
\D^-(\Mod R)\ar@{^(->}[u] \ar[rr]^{\sim} && \frac{\D^-(\Mod (\mmod R))}{\D^-_0(\Mod (\mmod R))} \ar@{^(->}[u]\\
\D^{\bb}(\Mod R) \ar[rr]^{\sim} \ar@{^(->}[u] && \frac{\D^{\bb}(\Mod (\mmod R))}{\D^{\bb}_0(\Mod (\mmod R))} \ar@{^(->}[u] }\]
of triangulated categories whose rows are triangle equivalences.
\end{scorollary}

\begin{proof}
The equivalence of the first row proved in  Theorem \ref{Thm1}, while the second and third equivalences follow directly from the definition of the functor $\eta: \D(\Mod R) \lrt \D_{R}(\Mod (\mmod R))$.
\end{proof}

The same method as in the proof of Lemma \ref{ExSeqCom} and Theorem \ref{Thm1} can be applied to get the following result. The equivalence of the bottom row has already been proved by Krause \cite[Corollary 3.2]{K}, but using a different approach.

\begin{sproposition}\label{Ext-Krause}
There is a commutative diagram
\[ \xymatrix{ \D(\mmod R) \ar[rr]^{\sim} && \frac{\D(\mmod (\mmod R))}{\D_0(\mmod (\mmod R))} \\
\D^-(\mmod R)\ar@{^(->}[u] \ar[rr]^{\sim} && \frac{\D^-(\mmod (\mmod R))}{\D^-_0(\mmod (\mmod R))}\ar@{^(->}[u]\\
\D^{\bb}(\mmod R) \ar[rr]^{\sim} \ar@{^(->}[u] && \frac{\D^{\bb}(\mmod (\mmod R))}{\D^{\bb}_0(\mmod (\mmod R))} \ar@{^(->}[u] }\]
of triangulated categories whose rows are triangle equivalences.
\end{sproposition}

\subsection{Covariant functors}\label{Covariant functors}
Gruson and Jensen \cite{GJ} characterized the injective objects of $\Mod(\mmod R)^{\op}$ as the functors isomorphic to those of the form $- \otimes_R M$, where $M$ is a pure-injective $R^{\op}$-module. Moreover, it is known \cite[Theorem B.16]{JL} that a covariant functor $F$ in $\Mod (\mmod R)^{\op}$ is fp-injective if and only if $F \cong -\otimes_R M$, for some (left) $R$-module $M$. Recall that a functor $F$ in $\Mod (\mmod R)^{\op}$ is fp-injective, if $\Ext^1_R(G,F)=0$, for all finitely presented functors $G$ in $\Mod (\mmod R)^{\op}$.
Let ${\rm fp}\mbox{-} \CI(R)^{\op}$ denote the full subcategory of $\Mod (\mmod R)^{\op}$ consisting of fp-injective objects. In fact, ${\rm fp}\mbox{-} \CI(R)^{\op} = (\mmod (\mmod R)^{\op})^\perp$.

Let $F$ be an object in $\Mod (\mmod R)^{\op}$. Then there is an injective copresentation
$$0 \lrt F \lrt - \otimes_R M \st{\varphi}\lrt -\otimes_R N$$
of $F$.
Since the functor $v: \Mod R \lrt \Mod (\mmod R)^{\op}$ defined by $v(M)= -\otimes_R M$ is full and faithful \cite[Theorem B.16]{JL}, there is a morphism $f: M \lrt N$ of $R$-modules such that $\varphi=- \otimes_R f$. If we set $L:= \Coker f$, then the above injective copresentation of $F$ can be completed to the following coresolution of $F$ by fp-injective objects
$$ 0 \lrt F \lrt -\otimes_R M \st{-\otimes f} \lrt - \otimes_R N  \lrt - \otimes_R L \lrt 0.$$
Hence, every functor $F$ in $\Mod (\mmod R)^{\op}$ admits a fp-injective coresolution of length at most 2.

On the other hand, we have the cotorsion theory $({}^\perp {\rm fp}\mbox{-} \CI(R)^{\op} , {\rm fp}\mbox{-} \CI(R)^{\op})$ in $\Mod (\mmod R)^{\op}$. Since $\mmod R$ is an essentially small category, this cotorsion theory  is cogenerated by a set. So it is complete and hence for every functor $F \in \Mod (\mmod R)^{\op}$, there exists a short exact sequence
$$ 0 \lrt F \lrt - \otimes_R M \lrt C \lrt 0,$$
where $M \in \Mod R^{\op}$ and $C \in {}^\perp {\rm fp}\mbox{-} \CI(R)^{\op}$.

A fp-injective coresolution
$$ 0 \lrt F \lrt -\otimes_R M \st{-\otimes f}\lrt -\otimes_R N \lrt -\otimes_R L \lrt 0 $$
of $F$ is called special if the image of $- \otimes f$ and $-\otimes_R L$ belong to ${}^\perp {\rm fp}\mbox{-} \CI(R)^{\op}$.
It follows from the above argument that each functor $F \in \Mod (\mmod R)^{\op}$ admits a special fp-injective coresolution of length 2.

Using this fact, in this subsection we provide a derived version of Auslander's formula for the category $\Mod (\mmod R)^{\op}$. The argument is similar, rather dual, to the argument we applied for the proof of Theorem \ref{Loc-Seq}. In fact, one should follow the following steps. We just give a sketch of proof for the first step. \\

\noindent {\bf Step $1$}. The functor $$v_{\centerdot}: (\Mod (\mmod R)^{\op}, \Mod R^{\op}) \lrt (\Mod R^{\op}, \Mod R^{\op}),$$ that is induced by the embedding $v: \Mod R^{\op} \lrt \Mod (\mmod R)^{\op}$, admits a right adjoint $v^{\centerdot}$ such that for all functors $F \in (\Mod R^{\op}, \Mod R^{\op})$, $v_{\centerdot}v^{\centerdot}F=F$. Moreover, if $F$ is exact, then $v^{\centerdot}F$ is an exact functor. Set $\vartheta:=v^{\centerdot}i: \Mod (\mmod R)^{\op} \lrt \Mod R^{\op}$, where $i: \Mod R^{\op} \lrt \Mod R^{\op}$.\\

\noindent {\bf Sketch of the proof of Step $1$.} Let $F:\Mod R^{\op} \lrt \mathcal{D}$ be a covariant functor. Define
$$ v^{\centerdot}F(T):= \Ker (T(M^0) \st{T(d^0)} \lrt T(M^1)),$$
where $0 \lrt T \lrt - \otimes_R M^0 \st{-\otimes d^0}\lrt -\otimes_RM^1 \st{-\otimes d^1}\lrt -\otimes_RM^2 \lrt 0 $ is a special fp-injective coresolution of $T$. Also, if there is a morphism $f: T \lrt T'$ of functors, it can be lifted to a morphism of their special fp-injective coresolutions and so we have a morphism $v^{\centerdot}F(f): v^{\centerdot}F(T) \lrt v^{\centerdot}F(T')$. Now, one can easily see that $v^\centerdot$ is the right  adjoint of $v_\centerdot$.\\

\noindent {\bf Step $2$}. Let ${\rm Mod}^0 \mbox{-} (\mmod R)^{\op}$ be the full subcategory of $\Mod(\mmod R)^{\op}$ formed by all functors $F$ with the property that $\vartheta(F)=0$. Since $\vartheta$ is an exact functor, ${\rm Mod}^0 \mbox{-}(\mmod R)^{\op}$ is a Serre subcategory of $\Mod (\mmod R)^{\op}$. Note that, by definition of the functor $v^{\centerdot}$, a functor $F \in \Mod (\mmod R)^{\op}$ belongs to ${\rm Mod}^0 \mbox{-} (\mmod R)^{\op}$ if and only if $F(R)=0$.\\

\noindent {\bf Step $3$}. For each functor $F$ in $\Mod(\mmod R)^{\op}$ there is a unique exact sequence
$$0 \lrt F_0 \lrt - \otimes \vartheta(F) \lrt F \lrt F_1 \lrt 0$$
such that $F_0, F_1 \in {\rm Mod}^0 \mbox{-}(\mmod  R)^{\op}$.\\

\noindent {\bf Step $4$}. $\Hom(- \otimes_R M, F)=0$, for all $M \in \Mod R$ and $F \in {\rm Mod}^0 \mbox{-} (\mmod R)^{\op}$.\\

Based on the above facts, similar to the proof of Theorem \ref{Loc-Seq}, one deduce that $(v,\vartheta)$ is an adjoint pair. So, we have the following colocalisation sequence of abelian categories
\[ \xymatrix@C=0.5cm@R=0.5cm{ {\rm Mod}^0 \mbox{-} (\mmod R)^{\op} \ar[rrr]  &&& \Mod (\mmod R)^{\op}  \ar[rrr]^{\vartheta}  \ar@/_1.5pc/[lll] &&& \Mod R^{\op}.  \ar@/_1.5pc/[lll]_{v} }\]
This in turn, implies that there is an equivalence $\Mod R^{\op}\simeq \frac{\Mod (\mmod R)^{\op}}{{\rm Mod}^0 \mbox{-} (\mmod R)^{\op}}$ of abelian categories.\\

On the other hand, since ${\rm Mod}^0 \mbox{-}(\mmod R)^{\op}$ is closed under direct sums, by \cite[Theorem 15.11]{F}, ${\rm Mod}^0 \mbox{-} (\mmod R)^{\op}$ is localizing, that is the quotient functor $q: \Mod (\mmod R)^{\op} \lrt \frac{\Mod (\mmod R)^{\op}}{{\rm Mod}^0 \mbox{-} (\mmod R)^{\op}}$ admits a right adjoint. Thus there exists the following localization sequence of abelian categories
\[ \xymatrix@C=0.5cm@R=0.5cm{ {\rm Mod}^0 \mbox{-} (\mmod R)^{\op} \ar[rrr]  &&& \Mod (\mmod R)^{\op}  \ar[rrr] \ar@/^1.5pc/[lll]  &&& \frac{\Mod (\mmod R)^{\op}}{{\rm Mod}^0 \mbox{-} (\mmod R)^{\op}}. \ar@/^1.5pc/[lll] }\]

\vspace{0.3cm}

So we have proved the following result.

\begin{sproposition}\label{Prop-cov}
Let $R$ be a right coherent ring. Then there is a recollement
\[ \xymatrix@C=0.5cm@R=0.5cm{ {\rm Mod}^0 \mbox{-} (\mmod R)^{\op} \ar[rrr]  &&& \Mod (\mmod R)^{\op}  \ar[rrr]^{\vartheta} \ar@/^1.5pc/[lll] \ar@/_1.5pc/[lll] &&& \Mod R^{\op} \ar@/^1.5pc/[lll] \ar@/_1.5pc/[lll]_{v} }\]

\vspace{0.3cm}

of abelian categories. In particular, $\Mod R^{\op}\simeq \frac{\Mod (\mmod R)^{\op}}{{\rm Mod}^0 \mbox{-} (\mmod R)^{\op}}$.
\end{sproposition}

Set $\D^0(\Mod(\mmod R)^{\op})$ to be the quotient category $\frac{\K_{R\mbox{-} \rm ac}(\Mod (\mmod R)^{\op})}{\K_{\ac}(\Mod (\mmod R)^{\op})}$. We also may follow the same argument, as in the case for contravariant functors and prove the following result. So we do not include a proof.

\begin{stheorem}
Let $R$ be a right coherent ring. Then there exists the following equivalence
$$ \D(\Mod R^{\op}) \simeq \frac{\D(\Mod (\mmod R)^{\op})}{\D^0(\Mod (\mmod R)^{\op})}$$
of triangulated categories.
\end{stheorem}

\section{Recollements involving  $\D(\Mod R)$}\label{Section 4}
Let $R$ be a right coherent ring. As  applications of our results in  Section \ref{Section 3}, here we provide recollements of homotopy category of pure-projective and homotopy category of pure-injective $R$-modules. These recollements are mixing together the pure exact structure with the usual exact structure.

\s A complex $\X \in \K(\Mod R)$ is called pure-exact if for every module $M \in \mmod R$, the induced complex $\Hom(M, \X)$ is exact.
Let $\K_{\rm pac}(\Mod R)$ denote  the full subcategory of $\K(\Mod R)$ formed by all pure-exact complexes. The pure derived category $\D_{\rm pur}(\Mod R)$, is the derived category with respect to the pure exact structure and so is the Verdier quotient $\K(\Mod R) / \K_{\rm pac}(\Mod R)$. In \cite{K12}, Krause studied this category and proved that $\D_{\rm pur}(\Mod R)$ is compactly generated with $\D_{\rm pur}(\Mod R)^{\rm c} \simeq \K^{\bb}(\mmod R)$. Moreover, he \cite[Corollary 6]{K12} proved that for a ring $R$, there exists a triangle equivalence
$$\D_{\rm pur}(\Mod R) \simeq \D(\Mod (\mmod R)).$$

\begin{remark} \label{Cot-F}

\begin{itemize}
\item [$(i)$] A functor $C \in \Mod (\mmod R)$ is called cotorsion if $\Ext^1(F,C)=0$, for all flat functors $F \in \CF(\mmod R)$. It is proved in \cite[Theorem 4]{H} that a flat functor $(-,M)$ in $\Mod(\mmod R)$ is cotorsion if and only if $M$ is a pure-injective module. Hence, the fully faithful functor $U: \Mod R \lrt \Mod (\mmod R)$ induces an equivalence ${\rm P}\Inj R \simeq {\rm Cot}\mbox{-} \CF(\mmod R)$, where ${\rm Cot\mbox{-}\CF}(\mmod R)$ denotes the full subcategory of $\CF(\mmod R)$ consisting of all cotorsion-flat functors. Moreover, the functor $U$ can be extended to the full and faithful functor $\K(U): \K(\Mod R) \lrt \K(\Mod (\mmod R))$ of triangulated categories. The above argument implies the following equivalence of triangulated categories $$\K({\rm P}\Inj R) \st{\st{\K(U)}\sim} \lrt \K({\rm Cot}\mbox{-} \CF(\mmod R)).$$
\item [$(ii)$]  As it is mentioned in \ref{ProjObj}, a functor $P$ in $\Mod (\mmod R)$ is projective if and only if $P \cong (-,M)$, for some pure-projective $R$-module $M$. So, there is an equivalence ${\rm P}\Prj R \simeq \CP(\mmod R)$ induced by the functor $U$. Furthermore, the full and faithful functor $\K(U): \K(\Mod R) \lrt \K(\Mod (\mmod R))$ restricts to an equivalence $$ \K({\rm P}\Prj R) \simeq \K(\CP(\mmod R))$$ of triangulated categories.
\item [(iii)] Let $\K_{\ac}(\CF(\mmod R))$ be the full triangulated subcategory of $\K(\CF(\mmod R))$ consisting of acyclic complexes of flat functors. We have a triangle equivalence $\K_{\pur}(\Mod R) \simeq \K_{\ac}(\CF(\mmod R))$ via the functor $\K(U)$.
\end{itemize}
\end{remark}

\begin{lemma}\label{(P,F)}
The pair $$(\K(\CP(\mmod R)), \K_{\rm ac}(\CF(\mmod R)))$$ is a stable $t$-structure in $\K(\CF(\mmod R))$. In particular, there is an equivalence
$$ \frac {\K(\CF(\mmod R))}{\K_{\rm ac}(\CF(\mmod R))}\simeq \K(\CP(\mmod R))$$
of triangulated categories.
\end{lemma}

\begin{proof}
In view of Theorem 5.4 of \cite{St1}, there exists the complete cotorsion theory $$(\C({\rm P}\Prj R), \C_{\pur}(\Mod R))$$ in $\C(\Mod R).$ Hence, by \cite[Theorem 3.5]{BEIJR}, the inclusion functor $\iota: \K({\rm P}\Prj R) \lrt \K(\Mod R)$ has a right adjoint $\iota_*: \K(\Mod R) \lrt \K({\rm P}\Prj R)$, which is defined as follows.
Since the above  cotorsion theory is complete, for each complex $\X \in \K(\Mod R)$, there is a short exact sequence
$$ 0\lrt {\bf D} \lrt {\bf C} \lrt \X \lrt 0$$
with ${\bf D} \in \C_{\pur}(\Mod R)$ and ${\bf C} \in \C({\rm P}\Prj R)$. Then, $\iota_*(\X)$ is defined to be the complex ${\bf C}$.  It follows directly from definition that the kernel of $\iota_*$ is the homotopy category $\K_{\pur}(\Mod R)$. So, Proposition \ref{Miyachi2} (iii) yields the stable $t$-structure $(\K({\rm P}\Prj R), \K_{\pur}(\Mod R))$ in $\K(\Mod R)$. Therefore, we have the stable $t$-structure $$(\K(\CP(\mmod R)) , \K_{\ac}(\CF(\mmod R)))$$ in $\K(\CF(\mmod R))$, see Remark \ref{Cot-F}(ii)-(iii).

Let $\beta: \K(\CF(\mmod R)) \lrt \K(\CP(\mmod R))$ be a triangle functor that makes the following diagram commutative
\[\xymatrix{ \K(\Mod R) \ar[r]^{\iota_*} \ar[d]^{\wr}_{\K(U)} & \K({\rm P}\Prj R)\ar[d]_{\wr}^{\K(U)} \\
\K(\CF(\mmod R)) \ar[r]^\beta & \K(\CP(\mmod R)).}\]
It can be easily checked, using the adjoint pair $(\iota , \iota_*)$, that $\beta$ is the right adjoint of the inclusion functor $\K(\CP(\mmod R)) \lrt \K(\CF(\mmod R))$.
Now, it follows from Proposition \ref{Miyachi2}(i) that the functor $\beta$ induces an equivalence
$$ \frac {\K(\CF(\mmod R))}{\K_{\rm ac}(\CF(\mmod R))}\simeq \K(\CP(\mmod R))$$
of triangulated categories.
\end{proof}

We also need the following parallel result.

\begin{lemma}\label{(F,C)}
The pair $$(\K_{\rm ac}(\CF(\mmod R)), \K({\rm Cot\mbox{-}\CF}(\mmod R)))$$ is a stable $t$-structure in $\K(\CF(\mmod R))$. In particular, there is an equivalence
$$ \frac{\K(\CF(\mmod R))}{\K_{\rm ac}(\CF(\mmod R))} \simeq \K(\CCF(\mmod R)).$$
of triangulated categories.
\end{lemma}

\begin{proof}
The proof is similar to the proof of the above lemma. Just note that by Theorem 5.4 of \cite{St1}, we have the following complete cotorsion theory
$$(\C_{\pur}(\Mod R), \C({\rm P}\Inj R))$$ in $\C(\Mod R)$. So, \cite[Theorem 3.5]{BEIJR} comes to play and implies that the inclusion functor $\iota: \K({\rm P}\Inj R) \lrt \K(\Mod R)$ possess a left adjoint $\iota^*: \K(\Mod R) \lrt \K({\rm P}\Inj R)$. The rest of the proof is similar. so we leave it as an easy exercise.
\end{proof}

Next proposition follows from Corollary 5.8 of \cite{St1}. The feature of the proof presented here is that we explicitly discuss the structure of the equivalences and will apply this structure to the forthcoming results. We preface the proposition with a remark.

\begin{remark}
In view of \cite[Theorem 4.2]{St2}, there exists the complete cotorsion pair $$(\C(\CF(\mmod R)), \C(\CF(\mmod R))^\perp)$$ in $\C(\Mod (\mmod R))$. Hence, by Theorem 3.5 of \cite{BEIJR}, there is a right adjoint $$\psi: \K(\Mod (\mmod R)) \lrt \K(\CF(\mmod R))$$ of the inclusion functor $\K(\CF(\mmod R)) \lrt \K(\Mod (\mmod R))$. The functor $\psi$ is defined as follows.
Let $\X$ be a complex in $\K(\Mod (\mmod R))$. Then there is a short exact sequence
$$ 0 \rt {\bf C} \rt  {\bf F} \rt \X \rt 0$$
where ${\bf F} \in \C(\CF(\mmod R))$ and ${\bf C} \in \C(\CF(\mmod R))^\perp$. Set $\psi(\X)= {\bf F}$. It follows from definition that $\C(\CF(\mmod R))^{\perp} \subseteq \C_{\ac}(\Mod (\mmod R))$. Let $\X \in \K(\Mod (\mmod R))$ be an acyclic complex and consider the corresponding short exact sequence  $ 0 \rt {\bf C} \rt  {\bf F} \rt \X \rt 0$, with ${\bf F} \in \C(\CF(\mmod R))$ and ${\bf C} \in \C(\CF(\mmod R))^\perp$. Since ${\bf C}$ is acyclic, ${\bf F}$ is acyclic as well. So, $\psi$ maps every acyclic complex in $\K(\Mod(\mmod R))$ to an acyclic, and hence pure-exact, complex in $\K(\CF(\mmod R))$.

Therefore, $\psi$ induces a triangle functor
$$\bar{\psi}: \D(\Mod (\mmod R)) \lrt \frac{\K(\CF(\mmod R))}{\K_{\ac}(\CF(\mmod R))}.$$
Moreover, the fully faithful functor $U: \Mod R \lrt \Mod (\mmod R)$ yields the following equivalence of triangulated categories
$$\chi: \D_{\rm pur}(\Mod R) \lrt \frac{\K(\CF(\mmod R))}{\K_{\ac}(\CF(\mmod R))}.$$
It is proved in \cite[Corollary 4.8]{K12} that the functor $U$ induces an equivalence
$$\eta: \D_{\rm pur}(\Mod R) \lrt \D(\Mod (\mmod R))$$
of triangulated categories. It can be easily checked that $\chi^{-1} \circ \bar{\psi}$ is the quasi-inverse of $\eta$. So $\bar{\psi}$ is an equivalence of triangulated categories.
\end{remark}

\begin{proposition}\label{isos}
Let $R$ be a right  coherent ring. Then there are the following triangle equivalences
\begin{itemize}
\item [$(i)$] $\D(\Mod (\mmod R)) \st{\Psi} \lrt \K({\rm P}\Inj R),$
\item [$(ii)$] $\D (\Mod (\mmod R)) \st{\Phi}\lrt \K({\rm P}\Prj R).$
\end{itemize}
\end{proposition}

\begin{proof}
$(i)$
By Lemma \ref{(F,C)} there is an equivalence
$$ \xi : \frac{\K(\CF(\mmod R))}{\K_{\rm ac}(\CF(\mmod R))} \lrt \K(\CCF(\mmod R))$$
given by $\xi(\X) = {\bf C}$, where ${\bf C}$ fits into a short exact sequence
$$0 \rt \X \rt {\bf C} \rt {\bf F} \rt 0$$
with ${\bf F} \in \C_{\ac}(\CF(\mmod R))$.

Combine the above equivalence together with the equivalence
$$\K(U)^{-1}: \K(\CCF(\mmod R)) \lrt \K(\CP(\mmod R)),$$ of part $(i)$ of Remark \ref{Cot-F},
we gain the following triangle equivalence
$$ \Psi: \D(\Mod (\mmod R)) \st{\bar{\psi}}\lrt \frac{\K(\CF(\mmod R))}{\K_{\ac}(\CF(\mmod R))} \st{\xi}\lrt
\K(\CCF(\mmod R)) \st{\K(U)^{-1}} \lrt \K({\rm P}\Inj R),$$
where $\bar{\psi}$ is the equivalence introduced in the above remark.

$(ii)$ This is similar to the proof of part $(i)$. One should apply Lemma \ref{(P,F)}, equivalence of Part $(ii)$ of Remark \ref{Cot-F},
and the above remark to get the following sequence of the equivalences
$$\Phi: \D(\Mod (\mmod R)) \st{\bar{\psi}} \lrt \frac{\K(\CF(\mmod R))}{\K_{\ac}(\CF(\mmod R))} \st{\beta} \lrt \K(\CP(\mmod R)) \st{\K(U)^{-1}}\lrt \K({\rm P}\Prj R)$$
of triangulated categories.
\end{proof}

\begin{theorem}\label{RecDer}
Let $R$ be a right coherent ring. Then the following statements hold true.
\begin{itemize}
\item [$(i)$] The equivalence $\Psi: \D(\Mod (\mmod R)) \lrt \K({\rm P}\Inj R)$ induces the following commutative diagram of recollements
\vspace{0.4cm}
\[ \xymatrix@C=0.3cm@R=0.4cm{\D_0(\Mod (\mmod R)) \ar[dd]^{\Psi|}\ar[rrr]  &&& \D(\Mod (\mmod R)) \ar[rrr] \ar[dd]^{\Psi} \ar@/_1.5pc/[lll] \ar@/^1.5pc/[lll] &&& \D_{R}(\Mod (\mmod R)) \ar[dd] \ar@/_1.5pc/[lll] \ar@/^1.5pc/[lll]
\\ \\  \K_{\ac}({\rm P}\Inj R)\ar[rrr]  &&& \K({\rm P}\Inj R)  \ar[rrr] \ar@/^1.5pc/[lll] \ar@/_1.5pc/[lll] &&& \frac{\K({\rm P}\Inj R)}{\K_{\ac}({\rm P}\Inj R)}, \ar@/^1.5pc/[lll] \ar@/_1.5pc/[lll] }\]
\vspace{0.3cm}
\\whose vertical functors are triangle equivalences. In particular, there is a recollement
\vspace{0.25cm}
\[ \xymatrix@C=0.5cm@R=0.5cm{ \K_{\ac}({\rm P}\Inj R)\ar[rrr]  &&& \K({\rm P}\Inj R)  \ar[rrr] \ar@/^1.5pc/[lll] \ar@/_1.5pc/[lll] &&& \D(\Mod R), \ar@/^1.5pc/[lll] \ar@/_1.5pc/[lll] }\]
\vspace{0.3cm}
\\of triangulated categories.
\item [$(ii)$] The equivalence $\Phi: \D(\Mod (\mmod R)) \lrt \K({\rm P}\Prj R)$ induces the following commutative diagram of recollements
\vspace{0.4cm}
\[ \xymatrix@C=0.3cm@R=0.4cm{\D_0(\Mod (\mmod R)) \ar[dd]^{\Phi|}\ar[rrr]  &&& \D(\Mod (\mmod R)) \ar[rrr] \ar[dd]^{\Phi} \ar@/_1.5pc/[lll] \ar@/^1.5pc/[lll] &&& \D_{R}(\Mod (\mmod R)) \ar[dd] \ar@/_1.5pc/[lll] \ar@/^1.5pc/[lll]
\\ \\  \K_{\ac}({\rm P}\Prj R)\ar[rrr]  &&& \K({\rm P}\Prj R)  \ar[rrr] \ar@/^1.5pc/[lll] \ar@/_1.5pc/[lll] &&& \frac{\K({\rm P}\Prj R)}{\K_{\ac}({\rm P}\Prj R)}, \ar@/^1.5pc/[lll] \ar@/_1.5pc/[lll] }\]
\vspace{0.3cm}
\\whose vertical functors are triangle equivalences. In particular, there is a recollement
\vspace{0.25cm}
\[ \xymatrix@C=0.5cm@R=0.5cm{ \K_{\ac}({\rm P}\Prj R)\ar[rrr]  &&& \K({\rm P}\Prj R)  \ar[rrr] \ar@/^1.5pc/[lll] \ar@/_1.5pc/[lll] &&& \D(\Mod R), \ar@/^1.5pc/[lll] \ar@/_1.5pc/[lll] }\]
\vspace{0.3cm}
\\of triangulated categories.
\end{itemize}
\end{theorem}

\begin{proof}
There exist stable $t$-structures
 $$(Q({}^\perp \K_{R\mbox{-} \ac}(\Mod (\mmod R))), \frac{\K_{R\mbox{-} \ac}(\Mod (\mmod R))}{\K_{\ac}(\Mod (\mmod R))}) \ \text{and}$$ $$ (\frac{\K_{R\mbox{-} \ac}(\Mod (\mmod R))}{\K_{\ac}(\Mod (\mmod R))}, Q(\K_{R\mbox{-} \ac}(\Mod (\mmod R))^\perp) ) \ \ \ \ $$
in $\frac{\K(\Mod (\mmod R))}{\K_{\ac}(\Mod (\mmod R))}$, where $Q: \K(\Mod (\mmod R)) \lrt \frac{\K(\Mod (\mmod R))}{\K_{\ac}(\Mod (\mmod R))}$ is the canonical functor.
Set $ \ \CU:= \Psi(Q({}^\perp \K_{R\mbox{-} \ac}(\Mod (\mmod R))))$, $\ \  \CV:= \Psi( \frac{\K_{R\mbox{-} \ac}(\Mod (\mmod R))}{\K_{\ac}(\Mod (\mmod R))}) \ $ \ and \ \ $ \ \ \ \CW:= \Psi (Q(\K_{R\mbox{-} \ac}(\Mod (\mmod R))^\perp))$. Since $\Psi$ is an equivalence, we have stable $t$-structures $(\CU, \CV)$ and $(\CV, \CW)$ in $\K({\rm P}\Inj R)$.
By definition, the functor $\Psi$ sends every complex $\X$ with the property that $\X(R)$ is acyclic to an acyclic complex of pure-injective $R$-modules. So, the equivalence
$$\Psi: \D(\Mod (\mmod R))\lrt \K({\rm P}\Inj R)$$
induces an equivalence $\Psi: \frac{\K_{R\mbox{-} \ac}(\Mod (\mmod R))}{\K_{\ac}(\Mod (\mmod R))} \lrt \K_{\ac}({\rm P}\Inj R)$ and so $$\Psi(\frac{\K_{R\mbox{-} \ac}(\Mod (\mmod R))}{\K_{\ac}(\Mod (\mmod R))}) = \K_{\ac}({\rm P}\Inj R).$$
Now, by Corollary 1.13 of \cite{IKM}, we have the desired commutative diagram of recollements.

For the second part, note that by Theorem \ref{Thm1}, $\D_{R}(\Mod(\mmod R)) \simeq \D(\Mod R)$ as triangulated categories. Thus, $\frac{\K({\rm P}\Inj R)}{\K_{\ac}({\rm P}\Inj R)} \simeq \D(\Mod R)$ and we have the desired recollement.

The same argument works to prove $(ii)$.
\end{proof}
As a direct consequence of the above theorem we have the following results.
\begin{corollary}
For a right coherent ring $R$, there is an equivalence
$$\K_{\ac}({\rm P}\Inj R) \simeq \K_{\ac}({\rm P}\Prj R)$$
of triangulated categories.
\end{corollary}

We need the following lemma for the proof of the next result.

\begin{lemma}\label{Murfet}\cite[Corolary  2.10]{Mu}
Let  there is the following recollement of triangulated categories with $\CT$ compactly generated
\vspace{0.2cm}
\[\xymatrix{\CT'\ar[rr] && \CT \ar[rr] \ar@/^1pc/[ll]\ar@/_1pc/[ll]&& \CT'' \ar@/^1pc/[ll] \ar@/_1pc/[ll] }\]
\vspace{0.1cm}
\\Then $\CT'$ is compactly generated, and if $\CT''$ is also compactly generated, then there is a triangle
equivalence up to direct summands $\frac{\CT^{\rm c}}{\CT''^{\rm c}} \st{\sim }\lrt \CT'^{\rm c}$.
\end{lemma}

\begin{corollary}\label{ComObj}
Let $R$ be a right coherent ring. Then the homotopy category $\K_{\ac}({\rm P}\Inj R)$ is compactly generated and there exists the following triangle equivalence up to direct summands
\[ \frac{\K^{\bb}(\mmod R)}{\K^{\bb}(\prj R)}\st{\sim} \lrt \K^{\rm c}_{\ac}({\rm P}\Inj R).\]
\end{corollary}
\begin{proof}
By Proposition \ref{isos}, there is an equivalence $\Psi: \D(\Mod (\mmod R)) \lrt \K({\rm P}\Inj R)$. View $\mmod R$, as a ring with several objects. The same argument as in the ring case, implies that $\D(\Mod (\mmod R))$ is compactly generated and $$\D(\Mod (\mmod R))^{\rm c} \simeq \K^{\bb}(\prj (\Mod (\mmod R))),$$ where $\prj (\Mod (\mmod R))$ denotes the full subcategory of $\Mod (\mmod R)$ consisting of finitely generated projective functors. Also, the Yoneda functor yields the equivalence $$\K^{\bb}(\prj (\Mod (\mmod R))) \simeq \K^{\bb}(\mmod R)$$ of triangulated categories.
Since $\Psi$ preserves direct sums, $\K({\rm P}\Inj R)$ is also compactly generated and $\K^{\rm c}({\rm P}\Inj R) \simeq \K^{\bb}(\mmod R)$.

Moreover, Theorem \ref{RecDer} gives the following recollement of triangulated categories
\vspace{0.3cm}
\[ \xymatrix@C=0.5cm@R=0.5cm{ \K_{\ac}({\rm P}\Inj R)\ar[rrr]  &&& \K({\rm P}\Inj R)  \ar[rrr] \ar@/^1.5pc/[lll] \ar@/_1.5pc/[lll] &&& \D(\Mod R). \ar@/^1.5pc/[lll] \ar@/_1.5pc/[lll] }\]
\vspace{0.2cm}
\\It is known that  $\D(\Mod R)$ is compactly generated and $\D^{\rm c}(\Mod R) \simeq \K^{\bb}(\prj R)$. Thus, we can apply Lemma \ref{Murfet} to get that $\K_{\ac}({\rm P}\Inj R)$ is compactly generated and there is a triangle equivalence up to direct summands
\[ \frac{\K^{\bb}(\mmod R)}{\K^{\bb}(\prj R)}\st{\sim} \lrt \K^{\rm c}_{\ac}({\rm P}\Inj R).\]
\end{proof}

\section*{Acknowledgments}
The authors also thank the Center of Excellence for Mathematics (University of Isfahan). Part of this work is carried out in IHES, Paris,
France, when the last author were visiting there and she would like to thank the support and excellent
atmosphere of IHES. This work was partially supported by a grant from the Simons Foundation.

\end{document}